\documentclass[a4paper]{article}

\usepackage{geometry}
\usepackage[centertags]{amsmath}
\usepackage{amsthm}
\usepackage{amssymb}
\usepackage[utf8]{inputenc}
\usepackage[colorlinks=true]{hyperref} 
\usepackage{framed}
\usepackage[T1]{fontenc}
\usepackage{todonotes}

\theoremstyle{definition}
\newtheorem{Def}{Definition}
\theoremstyle{plain}
\newtheorem{Lem}[Def]{Lemma}
\newtheorem{Cor}[Def]{Corollary}
\newtheorem{Pro}[Def]{Proposition}
\newtheorem{Teo}{Theorem}

\theoremstyle{remark}

\makeatletter
\@addtoreset{table}{section}
\makeatother

\newcommand{\ord}{\operatorname{ord}}

\newcommand{\lex}{\operatorname{lex}}

\newcommand{\Sing}{\operatorname{Sing}}

\newcommand{\mult}{\operatorname{mult}}

\newcommand{\R}{{\mathbb R }}
\newcommand{\F}{{\mathcal F }}
\newcommand{\Z}{\mathbb{Z}}
\newcommand{\Q}{\mathbb{Q}}
\newcommand{\N}{\mathbb{N}}

\renewcommand{\P}{\mathbb{P}}

\newcommand{\K}{{\mathcal K}}

\renewcommand{\O}{{\mathcal O}}

\renewcommand\hat{\widehat}
\newcommand{\NO}{{Newton--Okounkov }}
\newcommand{\deq}{\ensuremath{\stackrel{\textrm{def}}{=}}}

\newcommand{\Y}{{Y_\bullet}}

\makeatletter
\def\vv{\@ifnextchar[{\@withv}{\@withoutv}}
\def\@withv[#1]#2#3{v(#2,#3;#1)}
\def\@withoutv#1#2{v(#1,#2)}
\makeatother

\makeatletter
\def\mm{\@ifnextchar[{\@withm}{\@withoutm}}
\def\@withm[#1]#2#3{\mu_{#1}(#2,#3)}
\def\@withoutm#1#2{\hat\mu(#1,#2)}
\makeatother

\begin{document}

\author{Joaquim Ro\'e
\footnote{Partially supported by MTM 2013-40680-P (Spanish MICINN grant)
and 2014 SGR 114 (Catalan AGAUR grant).}}
\title{Local positivity in terms of Newton--Okounkov bodies}

\maketitle

\begin{abstract}
In recent years, the study of \NO bodies on normal varieties
has become a central subject in the asymptotic theory of linear series,
after its introduction by Lazarsfeld--Musta\c{t}\u{a} 
and Kaveh--Khovanskii.
One reason for this is that they encode all numerical equivalence
information of divisor classes (by work of Jow). At the same time,
they can be seen as local positivity invariants, and
K\"uronya-Lozovanu have studied them in depth from this
point of view.

We determine what information is encoded by the set of all \NO 
bodies of a big divisor with respect to flags centered at a fixed point 
of a surface, by showing that it determines and is determined by
the numerical equivalence class of the divisor
up to negative components in the Zariski decomposition that
do not go through the fixed point. 

\end{abstract}

\section{Introduction}
\paragraph {\NO bodies}
Inspired by the work of A.~Okounkov \cite{Ok}, R.~La\-zarsfeld and
M.~Musta\c{t}\u{a} \cite{LazMus09} and independently
K.~Kaveh and A.~Khovanskii \cite{KK09} introduced \NO bodies
as a tool in the asymptotic theory of linear series on normal varieties, a tool
which proved to be very powerful and in recent developments of the theory
has gained a central role. An excellent introduction to the subject
---not exhaustive due to the rapid development of the theory--- 
can be found in the review \cite{Bou13} by S.~Boucksom.

\NO bodies are defined as follows. Let $X$ be a normal projective variety of 
of dimension $n$. 
A flag of irreducible subvarieties
\[\Y=\left\{X=Y_0 \supset Y_{1} \supset \dots \supset Y_n=\{p\}\right\} \]
is called \emph{full} and \emph{admissible} if 
\(Y_i\) has codimension \(i\) in \(X\) and 
is smooth at the point $p$. \(p\) is called the \emph{center} of the flag.
For every  non-zero rational function $\phi\in K(X)$, write \(\phi_0=\phi\),
and for \(i=1,\dots,n\)
\begin{equation}
\label{eq:flag_valuation}
   \nu_i(\phi) \deq \ord_{Y_i}(\phi_{i-1})\ , \qquad
 \phi_i \deq \left.\frac{\phi_{i-1}}{g_i^{\nu_i(\phi)}}\right|_{Y_i},
\end{equation}
where $g_i$ is a local equation of $Y_i$ in \(Y_{i-1}\)  around $p$
(this makes sense because the flag is admissible). 
The sequence $\nu_\Y=(\nu_1,\nu_2,\dots,\nu_n)$ determines a rank \(n\) 
discrete valuation
\(K(X)^* \longrightarrow \Z^n_{\lex}\)
with center at \(p\)  \cite{ZS75}. 

\begin{Def}
If $X$ is a normal projective variety, $D$ a big Cartier divisor on it, and 
$\Y$ an admissible flag, the \NO
body of $D$ with respect to $\Y$ is
\[
\Delta_{\Y}(D) \deq\overline{\left\{\left.\frac{\nu_\Y(\phi)}{k}\right| 
\phi\in H^0(X,\O_X(kD)),\,k\in \N \right\}}\subset \R^n,
\] 
where $\overline{\{\cdot\}}$ denotes the closure with respect
to the usual topology of $\R^2$. Although not obvious from this definition,
\(\Delta_{\Y}(D)\) is convex and compact, with nonempty interior,
i.e., a body (see \cite{KK09}, \cite{LazMus09}, \cite{Bou13}). 
A.~K\"uronya, V.~Lozovanu and C.~Maclean \cite{KLM} have shown that
it is a polygon if \(X\) is a surface, and that in higher dimensions it can 
be non-polyhedral, even if \(X\) is a Mori dream space.
The definition can be carried over to 
the more general setting of graded linear series, 
and also to \(\Q\) or \(\R\)-divisors; in the absence
of some bigness condition, \(\Delta_{\Y}(D)\) may fail to have top dimension.

 \end{Def}
By  \cite[Proposition~4.1]{LazMus09}, $\Delta_{\Y}(D)$
only depends on the  numerical equivalence class of $D$. 
S.~Y.~Jow proved in \cite{Jow} that the set of all \NO bodies works as a complete
set of numerical invariants of $D$, in the sense that, if $D'$
is another big Cartier divisor with $\Delta_{\Y}(D)=\Delta_{\Y}(D')$
for all flags $\Y$, then $D$ and $D'$ are numerically equivalent.

\paragraph{Flags on proper models of $K(X)$}
It is most natural to define \NO bodies with respect to any valuation \(\nu\)
with value group equal to \(\Z^{\dim X}_{\lex}\), and not only those coming from flags
on \(X\) (see \cite{KK09}, \cite{Bou13}). 
Thus we consider
admissible flags on arbitrary birational models 
of \(X\), noting that even to express the results for flags lying on \(X\) 
(theorem \ref{teo:smooth} below) we
need to consider clusters of infinitely near points.

\begin{Def}
We call \emph{admissible flag for} $X$ any
admissible flag $\Y$ on \(\tilde X\) where
$\pi:\tilde X\rightarrow X$ is a proper birational morphism.
Whenever we need to specify the map we will use
the notation 
\[\Y=\left\{X\overset{\pi}\longleftarrow \tilde X=
Y_0 \supset Y_{1} \supset \dots \supset Y_n=\{p\}\right\} \]
but mostly we omit
an explicit mention of the model $\tilde X$ on which
$p$ and $Y_i$ lie. The point $\pi(p)=O\in X$ will be called the 
\emph{center} of the flag on $X$; if $\pi$ contracts the whole
flag, i.e., $\pi(Y_1)=\pi(p)=O$ then we 
say that $\Y$ is an \emph{infinitesimal} flag, and if
\(\operatorname{codim} \pi(Y_i)=i\) then it is a 
\emph{proper} flag. If \(\tilde X=X, \pi=id_X\), we say that
the flag is \emph{smooth} at \(O\). The corresponding \NO
bodies will be also called infinitesimal, proper or smooth
accordingly.
\end{Def}

Already Lazarsfeld--Musta\c{t}\u{a} 
\cite{LazMus09} considered 
\NO bodies of \(D\) defined by flags on varieties birational to \(X\)
---more precisely,  flags contained
in the exceptional divisor of a blowup of \(X\), with the goal of making a canonical 
choice of ``generic infinitesimal'' flag and getting rid of the arbitraryness
of the choice of a flag---. 
A.~K\"uronya and V.~Lozovanu \cite{KL} have pushed forward
the study of infinitesimal flags, with the philosophical viewpoint that the
``local positivity'' of \(D\) at a smooth point \(O\) should be governed by
the set of \NO bodies \(\Delta_{\Y}(D)\) where the flag \(\Y\)
is centered at \(O\). This raises the question of what information 
on \(L\) is contained
in the set of \NO bodies \(\Delta_{\Y}(D)\) with fixed center,
analogously to Jow's result for the set of all \NO bodies.
In the case when \(X\) is a surface, we provide a complete
answer which supports
 the ``local positivity'' viewpoint, and we prove that
\NO bodies given by infinitesimal flags suffice to determine all
\NO bodies given by flags centered at \(O\).

\paragraph{Clusters of infinitely near points}
Fix \(X\) a projective surface, and \(O\in X\) a smooth point.
A point infintely near to \(O\) is a smooth point \(p\in\tilde X\),
where \(\pi:\tilde X\rightarrow X\) is a proper birational morphism,
such that \(\pi(p)=O\). 

A finite or infinite set $K$ of points equal or infinitely near to 
$O$, such that for each $p\in K$, $K$ contains all points to which $p$ is
infinitely near, is called a \emph{cluster} of points infinitely near
to $O$. We now review a few facts on clusters that we need,
refering to E.~Casas-Alvero's book \cite{Cas00} for details and proofs.
The simplest example of a cluster is the sequence of images of
a point \(p \in \tilde X\) infinitely near to \(O\): \(\pi_p\) can be factored as a sequence
of $k$ point blowups \(\pi=bl_O\circ bl_{p_1}\circ\dots\circ bl_{p_{k-1}}\), 
\[ X=X_0\overset{bl_O}\longleftarrow
 X_1\overset{bl_{p_1}}\longleftarrow\dots\overset{bl_{p_{k-1}}}\longleftarrow
 X_k=\tilde X=X_p
 \]
and then
\[
 K(p)=\{O,bl_{p_1}\circ\dots\circ bl_{p_{k-1}}(p), \dots,
bl_{p_{k-1}}(p), p\}
\]
is a cluster. A priori, infinitely near points belong to
different surfaces, but we consider the points 
\(p\in X_p\overset{\pi_p}\longrightarrow X\) and 
\(p'\in X_{p'}\overset{\pi_{p'}}\longrightarrow X\) to be the same point
when there is a birational map defined in a neighborhood of \(p\),
\(X_p\supset U_p\rightarrow X_{p'}\), which commutes with 
\(\pi_p, \pi_{p'}\), maps \(p\) to \(p'\)
and is an isomorphism in a (possibly smaller) neighborhood of \(p\). 
Then we can safely assume
that the sequence of points blown up to get the surface where \(p\) lies 
is formed exactly by the points in \(K(p)\) except \(p\) itself:
\(K(p)=\{O,p_1,\dots,p_{k-1},p\}\). In this sense, every infinitely near point
\(p\) has a well defined predecessor, namely the last blown up point
\(p_{k-1}\).

Points infinitely near to \(O\) 
are classified as \emph{satellite} if \(p \in \Sing (\pi^{-1}(O)) \)
and \emph{free} otherwise.
We shall call a cluster 
\(K\) \emph{free} if every \(p \in K\) is free.
A relevant fact when dealing with smooth flags is that there is a smooth curve
through \(O\) whose birational transform in \(\tilde X\) contains 
the infinitely near point \(p\in \tilde X\)
if and only if the cluster \(K (p)\) is free.
It is customary to say that a curve goes through an infinitely near
point \(p\) (or has multiplicity \(m\) there) if its birational transform does so;
we will follow this convention without further notice.

A \emph{weighted cluster} is a pair \(\K=(K,m)\) 
where \(K\) is a cluster and \(m\) is a map
\(m:K\longrightarrow \Z\). A typical example is,
given a proper birational morphism \(\tilde X \overset\pi\longrightarrow X\) 
(factored as above)
and a curve \(C\) through \(O\),
 the set of all points infinitely near to \(O\)
in \(\bigcup_{i=0}^k X_i \) that belong to 
 \(C\), weighted with \(m(p)=\mult_p(\tilde C)\).

Let \(C\subset X\) be a curve through \(O\) which has no smooth 
branch through \(O\). There exists a minimal model
\(\pi:\tilde X \longrightarrow X\) such that, denoting \(\tilde C\)
the strict transform of \(C\), all of the (finitely many) 
points of \(\tilde C\) infinitely near to \(O\) (i.e., 
\(\pi^{-1}(O)\cap \tilde C\)) are satellite.
For any factorization of such a \(\pi\)
 as a sequence of point blowups, the centers of the blowups
 form a free cluster. This cluster, weighted with the multiplicities
 of  \(C\) at its points, will be called the
 cluster of initial free points of \(C\) and denoted \(\mathcal F_C\).
Remark that an equality \(\mathcal F_C=\mathcal F_{C'}\)
means that the minimal model such that the strict transform 
of \(C\) has no free point infinitely near to \(O\)
is also the minimal model such that the strict transform 
of \(C'\)  has no free point infinitely near to \(O\), and moreover the
multiplicities of the strict transforms of \(C\) and \(C'\) at each
blown up point coincide.

\paragraph{Local numerical equivalence on surfaces}
Let still \(X\) be a normal projective surface. 
Every pseudoeffective \(\Q\)-divisor \(D\)
admits a unique \emph{Zariski decomposition} 
\(
D=P+N,
\)
where \(P, N\) are \(\Q\)-divisors with \(P\) nef, \(N\) effective,
the components \(N_i\) of \(N\) have negative definite intersection matrix,
and \(P\cdot N_i=0\). 
Zariski showed in \cite{Zar62} that a unique such decomposition exists 
for any effective divisor \(D\) on a smooth surface ---in what can be considered
a foundational work of the asymptotic theory of linear systems.
The generalization to pseudoeffective \(\Q\)-divisors is due to Fujita \cite{Fuj79}.
The result then carries over to normal surfaces using the intersection theory 
developed by Sakai in \cite{Sak}, see \cite[Theorem 2.2]{Pro}. One should bear in mind
that in this case \(P\) and \(N\) are in general Weil divisors only, even 
if \(D\) is Cartier.

\begin{Def}
 Fix $O\in X$, and let $D$ be a divisor on $X$, with Zariski decomposition
\(
D=P+N.
\)
We decompose the negative part as
\[
N=N_O+N_O^c
\]
where the support of $N_O$ are exactly the divisors in $N$
which go through $O$. We say that 
\[
D=P+N_O+N_O^c
\]
is the \emph{refined Zariski decomposition at $O$}.
\end{Def}

\begin{Def}
Given two divisors \(D, D'\) on \(X\) with
refined Zariski decompositions at \(O\)
\[
D=P+N_O+N_O^c, \qquad D'=P'+N'_O+{N'}^c_O
\]
we say that $D$ and $D'$ are \emph{numerically equivalent near $O$}
if
\begin{equation}
\label{eq:num_equiv}
 P\equiv P' \text{ and } N_O=N'_O.
\end{equation}

\end{Def}

\paragraph{The main results}\hspace{-.5em}of this paper show that the information
contained in the set of all \NO bodies of a big Cartier divisor \(D\)
with center at a smooth point
\(O\) of a surface is exactly the numerical equivalence
class near \(O\) of \(D\) in the sense above. 
\begin{Teo} \label{teo:main}
Let $D, D'$ be big Cartier divisors on a normal projective surface $X$, 
and let $O\in X$ a smooth point.
 The following are equivalent:
\begin{enumerate}
 \item \label{teo:num_equiv}
 $D$ and $D'$ are numerically equivalent near $O$, i.e.,
 their Zariski decompositions satisfy \eqref{eq:num_equiv}.
 \item \label{teo:center}
 For all admissible flags with center \(O\),
$\Delta_{\Y}(D)=\Delta_{\Y}(D')$.
 \item \label{teo:infinitesimal}
 For all infinitesimal admissible flags with center \(O\),
$\Delta_{\Y}(D)=\Delta_{\Y}(D')$.
 \item \label{teo:curve}
 For all proper admissible  flags with center \(O\),
$\Delta_{\Y}(D)=\Delta_{\Y}(D')$.
\end{enumerate}
\end{Teo}

It is obvious that \eqref{teo:center} is equivalent to
 [\eqref{teo:infinitesimal} and \eqref{teo:curve}].  
 The skeleton of our proof is as follows:
\[ \eqref{teo:num_equiv} \Rightarrow \eqref{teo:center}, \quad
   \eqref{teo:infinitesimal} \Rightarrow \eqref{teo:curve}
 \Rightarrow \eqref{teo:num_equiv}. \]
 Each implication follows from one or two of the lemmas 
 in section \ref{sec:proofs};
 some of the lemmas are actually stronger than is required
 and may be interesting for themselves.

Remark that it is not enough to know the \NO bodies of \(D\) with
respect to all flags lying on \(X\) with center at \(O\) (smooth flags) 
in order to recover the numerical equivalence class near
\(O\). The information contained in this smaller collection of \NO bodies
is determined in the next theorem, after which it will be easy to give
examples.
Assume  \(D\)  is a divisor with refined Zariski decomposition
 \[ D=P+N_O+N^c_O ,\] and
decompose further \(N_O=N_O^{sing} + N_O^{sm}\)
where \(N_O^{sm}\) is formed by all components with at least
one smooth branch through \(O\). Then the result can be 
expressed in terms of the the clusters of initial free points 
 \(\mathcal F_{N_O^{sing}}\) and \(\mathcal F_{{N'}_O^{sing}}\):
\begin{Teo}\label{teo:smooth}
Fix $O\in X$ a smooth point. 
 Let $D, D'$ be big Cartier divisors on $X$, with refined Zariski decompositions
\[
 D=P+N_O^{sing}+N_O^{sm}+{N}^c_O, \quad
  D'=P'+{N'}_O^{sing}+{N'}_O^{sm}+{N'}^c_O
\]
 The following are equivalent:
\begin{enumerate}
 \item \label{teo:equiv_smooth}
  \(P\equiv P'\), \(N_O^{sm}={N'}_O^{sm}\) and
  \(\mathcal F_{N_O^{sing}}=\mathcal F_{{N'}_O^{sing}}\).
 \item \label{teo:infinitesimal_smooth}
 For almost all infinitesimal admissible flags \(\{\tilde X \supset E\supset \{p\}\}\)
 with center \(O\) such that the cluster \(K(p)\) is free,
$\Delta_{\Y}(D)=\Delta_{\Y}(D')$.
 \item \label{teo:curve_smooth}
 For all smooth admissible  flags with center \(O\),
$\Delta_{\Y}(D)=\Delta_{\Y}(D')$.
\end{enumerate}
\end{Teo}

The easiest example in which the set of all smooth \NO bodies
with center at \(O\) does not determine the numerical equivalence class
near \(O\) is given by two big Cartier divisors \(D, D'\) whose negative parts \(N,N'\)
are distinct irreducible curves with ordinary cusps at the same point 
\(O\) and with the same tangent direction (it is not difficult to construct
such divisors on suitable blowups of \(\P^2\)). In that case 
\(\mathcal F_{N_O^{sing}}=\mathcal F_{{N'}_O^{sing}}\) consists of 
two points: \(O\) and the point infinitely  near to it in the direction 
tangent to the cusps, with multiplicities 2 and 1 respectively. Therefore 
all \NO bodies with respect to smooth flags centered at \(O\) coincide,
but \(N_O=N\ne N'=N'_O\). 

The proof, contained in the lemmas of section \ref{sec:proofs}
follows the same structure as for theorem \ref{teo:main}. The main ingredient
in both cases is the computation of \NO bodies in terms of Zariski decompositions
which can be found as Theorem 6.4 in \cite{LazMus09}.
Although Lazarsfeld and Musta\c{t}a proved this fact for smooth surfaces,
the result applies on a normal surface $X$ as long as the flag is centered 
at a smooth point $O$ of $X$. Indeed, using a resolution of singularities 
$\pi:\tilde X \rightarrow X$ which is an isomorphism in a neighbourhood of $O$,
one may apply \cite[Theorem 6.4]{LazMus09} to the pullback divisor $\pi^*D$
with respect to the pulled back flag, because
Zariski decompositions agree via pull-backs (see \cite[2.3]{Pro})
and intersection numbers agree by the projection formula.

\paragraph{Higher dimension}
Our results depend on the existence of a Zariski decomposition.
Decompositions with similar properties exist on some higher dimensional
varieties as well (for instance, on toric varieties)  and in
that case one can expect the behaviour of \NO bodies to be related 
to the decomposition similarly to what happens for surfaces.

Given a Cartier \(\R\)-divisor \(D\) on \(X\), a Zariski decomposition of \(D\)
in the sense of Cutkosky-Kawamata-Moriwaki (or simply a CKM-Zariski
decomposition) is an equality
\[ \pi^*D = P + N \]
on a smooth birational modification 
\(\pi:\tilde X \rightarrow X\) such that 
\begin{enumerate}
 \item \(P\) is nef,
 \item \(N\) is effective,
 \item all sections of multiples of \(D\) are carried by \(P\), i.e., the natural
 maps 
 \[H^0(\tilde X,\mathcal O_{\tilde X}(\lfloor kP\rfloor))\longrightarrow 
 H^0(\tilde X,\mathcal O_{\tilde X}(\lfloor k\pi^*D\rfloor))=
 H^0( X,\mathcal O_{X}(\lfloor kD\rfloor))\]
 are bijective for all \(k\ge0\).
\end{enumerate}
See Y. Prokhorov's survey \cite{Pro} for more on  CKM-Zariski
decompositions and other generalizations. 
Such decompositions
don't always exist \cite{Cut} and when they do, \(P\) and \(N\) may be
irrational even if \(D\) is an integral divisor. 
But if they do exist, for instance if
\(X\) is a toric variety \cite{Kaw87},
\NO bodies centered at a given point \(O\) will be governed by the 
the Zariski decomposition:

\begin{Pro}\label{pro:highdim}
  Let $D, D'$ be big Cartier divisors on a variety $X$, 
admitting a CKM-Zariski decomposition
and let $O\in X$ a  point.
 If  $D$ and $D'$ are numerically equivalent near $O$, i.e.,
 their CKM-Zariski decompositions satisfy \eqref{eq:num_equiv}.
 Then for all admissible flags with center \(O\),
$\Delta_{\Y}(D)=\Delta_{\Y}(D')$.
\end{Pro}

It should be expected that a converse statement similar to what holds
for surfaces be valid in higher dimension. In fact, 
the proof of lemma \ref{lem:negative} below can be easily adapted to the higher
dimensional setting, so \(N_O\) is indeed determined by the 
\NO bodies centered at \(O\). The methods of this note are however not 
suffucient to show that the positive part is also determined by the 
\NO bodies centered at \(O\).

We work over an algebraically closed field. 

\paragraph{Acknowledgement} The author greatly benefited from conversations
with A.~K\"uronya on the contents of this work. 

\section{Proofs} \label{sec:proofs}
 \paragraph{Local numerical equivalence implies equal \NO bodies}
Let us first prove \(\eqref{teo:num_equiv} \Rightarrow \eqref{teo:center}\)
in theorems \ref{teo:main} and \ref{teo:smooth}, and at the same time
proposition \ref{pro:highdim}.
So assume that $D$ and $D'$ are big Cartier divisors on a variety \(X\),
numerically equivalent near $O$, i.e.,
with refined CKM-Zariski decompositions satisfying \eqref{eq:num_equiv}:
\[
\pi^*D=P+N_O+N_O^c, \qquad \pi^*D'=P'+N_O+{N'}^c_O
\]
with \( P\equiv P' \). 
By Lazarsfeld--Musta\c t\u a \cite[Theorem 4.1]{LazMus09}, 
all  Newton-Okounkov bodies of \(D'\) and
\[
D''=D'+(P-P')=P+N_O+{N'}^c_O
\]
coincide, because $D''$ and $D'$ are numerically equivalent.
Thus for the proof of proposition \ref{pro:highdim} and
\(\eqref{teo:num_equiv} \Rightarrow \eqref{teo:center}\)
in theorem \ref{teo:main} 
it is not restrictive to assume that $P=P'$. 
Then there is a sequence of divisors
\[
D=D_0, D_1, \dots, D_k=D'
\]
whose CKM-Zariski decompositions $\pi^*D_i=P+N_i$ have the same positive
part and $N_i$ differs from $N_{i+1}$ in a  multiple of a 
base divisor $E_i$
with $O\not\in \pi(E_i)$. Thus the desired implication follows from the following:
\begin{Lem}
 Let \(D,D'\) be two big Cartier divisors with respective refined Zariski decompositions
 \[
\pi^*D=P+N, \quad \pi^*D'=P+(N+\lambda E)
\]
with $\lambda\in \R$, and \(O\) any point $O\not\in \pi(E)$. 
Then for all admissible flags with center \(O\),
$\Delta_{\Y}(D)=\Delta_{\Y}(D')$.
\end{Lem}

\begin{proof}
An equation of $E$ is invertible
in a neighborhood of $O$, and therefore also in a neighborhood
of every point \(p\) infinitely near to \(O\). 
So for every flag $\Y$ centered at
$O$, $\nu_\Y(E)=0$. Since global sections of $\lfloor kD\rfloor$ and 
$\lfloor kD'\rfloor$
differ exactly in $\lfloor k\lambda E\rfloor$, their values under $\nu_\Y$ 
agree, and therefore the \NO bodies are the same.
 \end{proof}

Now the following lemma is enough to finish the proof of 
\(\eqref{teo:num_equiv} \Rightarrow \eqref{teo:center}\)
in theorem \ref{teo:smooth}:
\begin{Lem}
 Let \(D,D'\) be two big Cartier divisors on a normal surface \(X\)
  with respective refined Zariski decompositions
 \[
D=P+(N+\lambda C), \quad D'=P+(N+\lambda' C')
\]
with $\lambda,\lambda'\in \R$, and $C, C'$ irreducible curves
 with no smooth branch through 
\(O\), and satisfying \(\lambda\F_C=\lambda'\F_{C'}\)
Then for all infinitesimal admissible flags 
\(\{X\overset{\pi}\longleftarrow \tilde X \supset E\supset \{p\}\}\)
 with center at a given smooth point \(O\) such that the cluster \(K(p)\) is free
 and  \(p\not\in \F_{C}\),
$\Delta_{\Y}(D)=\Delta_{\Y}(D')$.
\end{Lem}

Note that the cluster \(\F_C=\F_{C'}\) is finite, and its weights 
are the multiplicities of \(C\) (equivalently, \(C'\)
at each \(p\in\F_C\). The equality \(\lambda\F_C=\lambda'\F_{C'}\)
means that both clusters consist of the same points, and their 
respective weights \(m, m'\) satisfy the proportionality
\(\lambda m(p)=\lambda'm'(p)\) for all \(p \in\F_C\).

\begin{proof}
Let 
\(\{X\overset{\pi}\longleftarrow \tilde X \supset E\supset \{p\}\}\)
be an infinitesimal admissible flag with center \(O\) 
such that the cluster \(K(p)=\{O=p_0,p_1,\dots,p_{k-1},p_k=p\}\) is free
 and  \(p\not\in \F_{C}\).
Let \(E'\) be the birational transform  of \(E\) in 
the blowup \(X_p\):
\[ X=X_0\overset{bl_O}\longleftarrow
 X_1\overset{bl_{p_1}}\longleftarrow\dots\overset{bl_{p_{k-1}}}\longleftarrow
 X_k=X_p.
 \]
Since \(p\in E'\), \(E'\) is an irreducible curve (is not contracted in \(X_p\)), 
and since \(E'\) contracts to the smooth point
\(O\), it must be one of the exceptional components;  
in fact it must be the last, \(E'=E_{p_{k-1}}\), which is the only one containing \(p\).
Thus it is not restrictive to assume that \(\tilde X=X_p\),
\(\pi=bl_O\circ bl_{p_1}\circ\dots\circ bl_{p_{k-1}}\), and \(E=E_{p_{k-1}}\).

The order of vanishing of \(\pi^*C\) along \(E\) is 
\[
\ord_E (\pi^*C)=\sum_{i=0}^{k-1} \mult_{p_i} \tilde C = 
\sum_{i=0}^{k-1} m(p_i)
\]
(where \(m(p_i)=0\) if \(p_i\not\in \F_C\)) because the \(p_i\) are free; 
similarly, \(\ord_E (\pi^*C')=\sum_{i=0}^{k-1} m'(p_i)\).
Moreover,
 \(C\) and \(C'\) do not pass through \(p\). 
So \(\lambda\nu_{Y_\bullet}(C)=\lambda'\nu_{Y_\bullet}(C')\). 
Therefore we conclude
as in the previous lemma: 
since global sections of $\lfloor kD\rfloor$ and 
$\lfloor kD'\rfloor$
differ exactly in $\lfloor k\lambda (C-C')\rfloor$, their values under $\nu_\Y$ 
agree, and therefore the \NO bodies are the same.
\end{proof}

\paragraph{Equality of infinitesimal bodies implies equality of proper bodies}
Now we prove \( \eqref{teo:infinitesimal} \Rightarrow \eqref{teo:curve} \)
in theorem \ref{teo:main} and 
\(\eqref{teo:infinitesimal_smooth}\Rightarrow \eqref{teo:curve_smooth}\)
in theorem \ref{teo:smooth}, namely we need to show that 
if $\Delta_{\Y}(D)=\Delta_{\Y}(D')$
 for all infinitesimal admissible flags \(\Y\) with center \(O\)
 (resp. for almost all infinitesimal admissible flags 
 \(\{\tilde X \supset E\supset \{p\}\}\)
 with center \(O\) such that the cluster \(K(p)\) is free)
then the same equality holds for
all proper admissible  flags \(\Y\) with center \(O\),
(resp. for all smooth admissible  flags \(\Y\) with center \(O\)).

Given a curve \(C\) through \(O\),
an infinite cluster \(K=\{p_0=O, p_1, \dots, p_k, \dots\}\) will be called
a \emph{branch cluster} for \(C\) if 
each \(p_i\) is infinitely near to \(p_{i-1}\) and all of them belong to
 \(C\). Note that in a branch cluster, at most
finitely many points are satellite, and  \(C\) has a smooth branch 
at \(O\) if and only if it admits a branch cluster which is free.

 Associated to each branch cluster there is a sequence of flags
\begin{equation}
 \label{eq:seq_flags}
 Y_{\bullet}^{(k)}=\{X_k \supset E_{p_{k-1}}\supset \{p_k\}\},  
\end{equation}
 and a corresponding sequence of valuations 
 $\nu^{(k)}=\nu_{Y_\bullet^{(k)}}$.
\begin{Lem}
Let \((\nu^{(k)})_{k\in\N}\) be the valuations associated to a branch cluster for
the irreducible curve \(C\) through \(O\). Let \(k_0\) be such that 
the birational transform
of \(C\) at \(p_{k_0}\) is smooth, and let
 $\Y=\{\tilde X \supset \tilde C \supset \{p_k\}\}$ be the corresponding
 proper admissible flag. Then
for every $\phi \in K(X)$  and every \(k\gg 0\) there is an equality
 \[
 \nu^{(k)}(\phi)=\begin{pmatrix}
 k-k_0&1\\1& 0
\end{pmatrix} \nu_{\Y}(\phi).
 \]
 \end{Lem}

\begin{proof}
Assume without loss of generality that \(\phi\) is a regular 
function on a neighbourhood of \(p_{k_0}\).
 Recall the definition of \(\nu_{\Y}(\phi)\):
  \(\nu_1(\phi)=\ord_C(\phi)\), \(\phi_1=\phi/g^{\nu_1(\phi)}\),
  where \(g\) is a local equation of \(\tilde C\) at \(p_{k_0}\),
  and  \(\nu_2(\phi)=\ord_{p_{k_0}}(\phi_1|_C)\). \(\phi_1\)
  is the local equation of some effective divisor \(D\) which does not
  contain \(C\), and hence, by Noether's formula
  for intersection multiplicities, \cite[Theorem 3.3.1]{Cas00},
  there is \(k_1\) such that \(D\) does not go through any point \(p_k, k\ge k_1\)
  and \(\nu_2(\phi)=\sum_{i=k_0}^{k_1}\mult_{p_i}D=\ord_{E_{p_{k-1}}}\phi_1\)
  for all \(k\ge k_1\). 
On the other hand, it is immediate that \(\ord_{E_{p_{k-1}}}g=k-k_0\)
for all \(k \ge \max\{k_0,1\}\), so 
\(\nu^{(k)}_1(\phi)=\ord_{E_{p_{k-1}}}(\phi)=(k-k_0)\nu_1(\phi)+\nu_2(\phi)\)
for all \(k\ge \max\{k_1,1\}\). Similarly, \(\nu^{(k)}_1(\phi_1)=0\) for \(k\ge k_1\)
and \(\nu^{(k)}_1(g)=1\), and the claim follows.
\end{proof}
\begin{Cor}\label{cor:infiniteness}
 Let \(D, D\) be arbitrary Cartier divisors on a surface \(X\), and \(O\in X\)
 a smooth point.
 Given a proper admissible flag
  \(Y_\bullet=\{X\overset{\pi}\longleftarrow\tilde X \supset C \supset \{p\}\}\) 
  centered at \(O\), and a branch cluster \(K=\{p_,\dots,p_k,\dots\}\) 
  for \(\pi(C)\),
  denote \(Y_{\bullet}^{(k)}\) the sequence of flags 
  \eqref{eq:seq_flags}. If the the set of indices \(k\) with 
  \(\Delta_{Y_{\bullet}^{(k)}}(D)=\Delta_{Y_{\bullet}^{(k)}}(D')\) is infinite, 
  then \(\Delta_{Y_{\bullet}}(D)=\Delta_{Y_{\bullet}}(D')\).
\end{Cor}

The proof of the corollary is straightforward and is left to the reader.

Now the desired implications in theorem \ref{teo:main} and 
\ref{teo:smooth} follow, because every curve \(C\) through \(O\) 
(resp. smooth at \(O\))
admits a branch cluster (resp. a free branch cluster), and statement
\eqref{teo:infinitesimal} in theorem \ref{teo:main} (resp.
\eqref{teo:infinitesimal_smooth} in theorem \ref{teo:smooth}) 
imply the infiniteness needed in corollary \ref{cor:infiniteness}.

\paragraph{Equality of proper bodies implies local numerical equivalence}
Finally we prove \( \eqref{teo:curve} \Rightarrow \eqref{teo:num_equiv} \)
in theorem \ref{teo:main} and 
\(\eqref{teo:curve_smooth}\Rightarrow \eqref{teo:equiv_smooth}\)
in theorem \ref{teo:smooth}.
We deal separately with the positive and negative parts, because for
the positive part it is enough to consider smooth proper flags:

\begin{Lem}
 Let \(D\) and \(D'\) be big Cartier divisors with Zariski decompositions
\[ D=P+N, \quad D'=P'+N'. \]
 Assume that, for all curves
 \(C\subset X\) smooth at \(O\), the bodies
 \(\Delta_\Y(D)=\Delta_\Y(D')\) agree for the flag 
 \(\Y=\{X\supset C\supset \{O\}\}\). Then \(P\equiv P'\).
\end{Lem}
\begin{proof}
 Choose ample divisor classes $L_1, \dots, L_\rho$ whose $\Q$-span
 is all of the rational N\'eron-Severi space $N^1(X)_\Q$.
 Replacing each $L_i$ by a suitable multiple, we can assume that
 it is the class of an irreducible curve $C_i$ smooth at $O$,
 whose tangent direction there is different from every tangent direction
 to a component of the augmented base locus which may pass
 through \(O\). 
 (This is well known, and can be proved as follows: 
 by Serre vanishing there exist $k$ such that 
 $H^1(X,\mathcal I_O^2 ( L_i^{\otimes k}))=0$
 where $\mathcal I_O$ denotes the ideal sheaf of the point $O$ in $X$.
 Then the exact sequence in cohomology determined by 
 \[ 0 \longrightarrow \mathcal I_O^2 \otimes L_i^{\otimes k} \longrightarrow
 \mathcal O_X ( L_i^{\otimes k}) \longrightarrow
 (\mathcal O_X/\mathcal I^2_O) ( L_i^{\otimes k}) \longrightarrow 0\]
 shows that $H^0(X,\O_X ( L_i^{\otimes k}))$ surjects onto
 $H^0(X,\mathcal O_X/\mathcal I^2_O)$ and in particular it is possible
 to find a section in $H^0(X,O_X ( L_i^{\otimes k}))$ which vanishes
 at $O$ to order exactly 1 and has preassigned image=tangent direction). 
 
 So for each $i=1,\dots,\rho$, we can compute \NO bodies
 of $D$ and $D'$ with respect to the flag 
 $\Y^{(i)}=\{ X \supset C_i \supset \{O\}\}$, 
 and by
 \cite[Theorem 6.4]{LazMus09}, the height of the intersection of 
 $\Delta_{\Y^{(i)}}(D)$
 with the second coordinate axis
 equals $P\cdot L_i$; since by hypothesis the bodies
 $\Delta_{\Y^{(i)}}(D)=\Delta_{\Y^{(i)}}(D')$
 coincide, it follows that $P\cdot L_i=P'\cdot L_i$ for all \(i\). 
 Therefore $P\equiv P'$.
\end{proof}

\begin{Lem}\label{lem:negative}
  Let \(D\) and \(D'\) be big Cartier divisors with refined Zariski decompositions
 \[ D=P+N_O+N^c_O, \quad D'=P'+N'_O+{N'}^c_O, \]
and assume that
 for all proper admissible  flags with center \(O\),
$\Delta_{\Y}(D)=\Delta_{\Y}(D')$. Then \(N_O=N'_O\).
  \end{Lem}
\begin{proof}
Let $C$ be a component of $N_O$, and let 
 $\pi:\tilde X\longrightarrow X$ be a proper birational morphism 
 such that the strict transform $\tilde C$ of $C$ is nonsingular
 at $\pi^{-1}(O)$; let $p$ a point in $\tilde C\cap \pi^{-1}(O)$,
 and $\Y=p\in \tilde C \subset \tilde X$, which is a proper
 admissible flag with center at \(O\). The first coordinate
 of the leftmost point in $\Delta_{\Y}(D)$ is the
 coefficient of $C$ in $N_O$ by the proof of  \cite[Theorem 6.4]{LazMus09},
 so this is also the coefficient of $C$ in $N'_O$. Doing this for each component,
 we obtain $N_O=N'_O$ as claimed.
\end{proof}
This finishes the proof of \( \eqref{teo:curve} \Rightarrow \eqref{teo:num_equiv} \)
in theorem \ref{teo:main}; to conclude with  
\(\eqref{teo:curve_smooth}\Rightarrow \eqref{teo:equiv_smooth}\)
in theorem \ref{teo:smooth} we need another lemma:

\begin{Lem}
Fix $O\in X$ a smooth point. 
 Let $D, D'$ be big Cartier divisors on $X$, with refined Zariski decompositions
\[
 D=P+N_O^{sing}+N_O^{sm}+{N}^c_O, \quad
  D'=P'+{N'}_O^{sing}+{N'}_O^{sm}+{N'}^c_O,
\]
and assume that
 for all smooth admissible  flags with center \(O\),
$\Delta_{\Y}(D)=\Delta_{\Y}(D')$. Then \(N_O^{sm}={N'}_O^{sm}\) and
  \(\mathcal F_{N_O^{sing}}=\mathcal F_{{N'}_O^{sing}}\).

  \end{Lem}
\begin{proof}
The equality  \(N_O^{sm}={N'}_O^{sm}\) follows with the same proof of
the previous lemma. Let us now show that   
\(\mathcal F_{N_O^{sing}}=\mathcal F_{{N'}_O^{sing}}\). 
Without loss of generality we may assume that  \(N_O^{sm}={N'}_O^{sm}=0\).

For each point \(p\in F_{N_O^{sing}}\), 
let \(C_p\) be a smooth curve going through
\(p\) and missing all points in \(F_{N_O^{sing}} \cup F_{{N'}_O^{sing}}\)
infinitely near to \(p\), and let \(Y_\bullet=\{X\supset C\supset\{O\}\). 
By \cite[Theorem 6.4]{LazMus09}, the least \(\alpha\)
such that \((0,\alpha)\in \Delta_{Y_\bullet}\) is \(\ord_O({N_O^{sing}}|_{C_p}\).
So the hypothesis tells us that 
\begin{equation}
 \label{eq:int_from_NO}
 \ord_O({N_O^{sing}}|_{C_p}=\ord_O({{N'}_O^{sing}}|_{C_p}
\end{equation}
If \(p=O\), then the two sides of the preceding inequality equal the weights of
\(O\) in \(\mathcal F_{N_O^{sing}}\) and \(\mathcal F_{{N'}_O^{sing}}\)
respectively.

If \(p\ne O\), let \(q\) be the point preceding \(p\) and \(C_q\) the corresponding
curve. The weights of \(p\) in \(\mathcal F_{N_O^{sing}}\) and \(\mathcal F_{{N'}_O^{sing}}\) now equal  the differences
\begin{gather*}
 m(p)=\ord_O({N_O^{sing}}|_{C_p})\,-\,\ord_O({N_O^{sing}}|_{C_q}) ,\\
m'(p)=\ord_O({{N'}_O^{sing}}|_{C_p})\,-\, \ord_O({{N'}_O^{sing}}|_{C_q}) 
\end{gather*}
respectively
 so they also coincide.
 
 We have proved that every point of the cluster \(\mathcal F_{N_O^{sing}}\) 
 appears in \(\mathcal F_{{N'}_O^{sing}}\) with the same weight; by symmetry,
 we conclude that both weighted clusters are in fact equal.
\end{proof}

\bigskip
\footnotesize
  \textsc{Departament de Matem\`atiques, 
  Universitat Aut\`onoma de Barcelona, 08193 Bellaterra (Barcelona)
Spain} \par  
  \textsc{Barcelona Graduate School of Mathematics} \par
  \textit{E-mail address}, \texttt{jroe@mat.uab.cat} 
\end{document}